\documentclass[11pt,oneside,leqno]{amsart}
\usepackage{amsxtra}
\usepackage{amsopn}
\usepackage{color}
\usepackage{amsmath,amsthm,amssymb}
\usepackage{amscd}
\usepackage{amsfonts}
\usepackage{latexsym}
\usepackage{verbatim}

\definecolor{verde}{rgb}{0.5,.7,.2}

\theoremstyle{plain}
\newtheorem{theorem}{Theorem}[section]
\newtheorem{definition}[theorem]{Definition}
\newtheorem{lemma}[theorem]{Lemma}
\newtheorem{proposition}[theorem]{Proposition}
\newtheorem{corollary}[theorem]{Corollary}
\newtheorem{remark}[theorem]{Remark}
\newtheorem{example}[theorem]{Example}
\newtheorem{question}[theorem]{QUESTION}
\newtheorem{remark-question}[section]{Remark-Question}
\newtheorem{conjecture}[section]{Conjecture}

\newcommand\C{{\mathbb C}}

\newcommand\R{{\mathbb R}}

\newcommand{\ov}[1]{\overline{ #1}}

\title[On the existence of  balanced and SKT metrics]{On the existence of  balanced and SKT metrics on nilmanifolds}
\author{Anna Fino and Luigi Vezzoni}
\address{Dipartimento di Matematica G. Peano \\ Universit\`a di Torino\\
Via Carlo Alberto 10\\
10123 Torino\\ Italy}
\email{annamaria.fino@unito.it}
 \email{luigi.vezzoni@unito.it}

\thanks{This work was partially supported by the project PRIN {\em Variet\`a reali e complesse: geometria, topologia e analisi armonica},  the project FIRB {\em Differential Geometry and Geometric functions theory} and by GNSAGA (Indam) of Italy.
}

\begin{document}

\begin{abstract} On a complex manifold  an Hermitian metric which is simultaneously SKT and balanced has to be necessarily K\"ahler.  It  has been conjectured that  if a  compact complex manifold  $(M, J)$ has an   SKT metric and    a  balanced metric both compatible with $J$, then $(M, J)$   is necessarily K\"ahler.  
We show that  the conjecture is true for  nilmanifolds.
\end{abstract}

\maketitle

\section{Introduction}

A  Riemannian metric $g$  on a complex manifold $(M,J)$ is {\em compatible} with $J$  (or {\em $J$-Hermitian})  if  $g (J \cdot, J\cdot ) = g (\cdot, \cdot)$.
In the present paper we focus on the existence of {\em special} Hermitian metrics on complex manifolds. More precisely,  we study the existence of {\em strong K\"ahler with torsion} (shortly SKT) and balanced metrics compatible with the same complex structure. 
We recall that a  $J$-Hermitian metric    is called {\em SKT} (or {\em pluriclosed}) if its fundamental form $\omega$ satisfies 
$$
\partial \bar \partial \omega=0\,,
$$
while $g$ is called {\em balanced} if $\omega$ is co-closed, i.e. 
$$
d^*\omega=0,
$$
where $d^*$ denotes the formal adjoint operator of $d$ with respect to the metric $g$.

SKT metrics were introduced by Bismut in \cite{Bismut} and further studied in many papers (see e.g. \cite{FV,GGPoon,FT,LFY,ST,ST1,Verb}  and the references therein), while 
balanced metrics were introduced and firstly studied by Michelsohn in \cite{M}, where their existence is characterized in terms of currents. In a subsequent paper Alessandrini and Bassanelli 
proved that modifications of compact balanced manifolds are always balanced (see \cite{AB1,AB2}) showing a powerful tool for finding examples of balanced manifolds. 

It is well-known that if an Hermitian metric $g$ is simultaneously SKT and balanced, then it is necessarily K\"ahler (see e.g. \cite{Ivanov}).   This result has been generalized in  \cite{IP}  showing that  a compact  SKT conformally balanced  manifold has to be K\"ahler. The next conjecture was stated in \cite{FV} and it is about the existence of an SKT metric and a balanced metric both compatible with the same complex structure:

\medskip
\noindent {\bf Conjecture.} {\em Every compact complex manifold admitting both an SKT metric and a balanced metric is necessarily K\"ahler.} 

\medskip
The conjecture has been implicitly already proved in literature in some special cases. For instance,
Verbitsky has showed  in \cite{Verb} that the twistor space of a compact, anti-self-dual Riemannian manifold has
no SKT metrics unless it is K\"ahlerian and Chiose has obtained in \cite{Chiose} a similar result for non-K\"ahler
complex manifolds belonging to the Fujiki class. Furthermore, Li, Fu and Yau have proved in \cite{LFY} that some new examples of SKT manifolds do not admit any balanced metric.
A natural  source of non-K\"ahler manifolds admitting balanced metrics and SKT metrics are given by  {\em nilmanifolds}, i.e.   by compact manifolds obtained as quotients
of a  simply connected nilpotent Lie group $G$ by a co-compact lattice $\Gamma$.   It is well known that a nilmanifold  cannot admit K\"ahler  structures unless it a torus (see for instance \cite{BG,Hasegawa}).
The aim of the paper is to show that the conjecture is true when $(M,J)$ is a {\em complex nilmanifold}.  By {\em complex nilmanifold} we refer to   a nilmanifold equipped with an {\em invariant} complex structure $J$,  i.e. endowed  with a complex structure  induced  by a left-invariant complex structure on $G$. 

Our result is the following        

\begin{theorem}\label{main}
Let $M=G\backslash\Gamma$ be a nilmanifold equipped with an invariant complex structure $J$.  Assume that $(M,J)$ admits a balanced metric $g$ and an SKT metric $g'$ both compatible with $J$. Then $(M,J)$  is a complex torus. 
\end{theorem}
The theorem is trivial in dimension $6$ in view of the main result in \cite{FPS} and it was already proved in \cite{FV} when the nilmanifold has dimension $8$ by using a classification result proven in \cite{EFV}. 

\section{Proof of Theorem \ref{main}}

In order to prove  Theorem \ref{main}   we need the following  lemmas
\begin{lemma}\label{2.1}
Let $(M=G/\Gamma,J)$ be a complex nilmanifold. 
\begin{itemize}
\item If $(M,J)$ has a balanced metric, then it has also an {\em invariant} balanced metric \cite{FG}. 

\item If $(M,J)$ has an SKT metric, then it has also an {\em invariant} SKT metric \cite{ugarte}.
\end{itemize}
\end{lemma}

\begin{lemma}[\cite{EFV}]\label{2.3}
Let $(M=G/\Gamma,J)$ be a complex nilmanifold of real  dimension $2n$. If $(M,J)$ has an SKT metric, then $G$ is (at most) $2$-step nilpotent, and there exists a complex $(1,0)$-coframe $\{\alpha^1,\dots,\alpha^n\}$ on $\mathfrak g$ satisfying the following structure equations$$
\left \{ \begin{array}{l}
d \alpha^j= 0, \quad j = 1, \ldots, k,\\
d \alpha^j  = \sum_{r, s =1}^k  \left(\frac12 c_{rs}^j  \alpha^{r}\wedge \alpha^s + c_{r \overline {s}}^j   \alpha^{r}\wedge \bar \alpha^{s}\right),\, j=k+1,\dots,n \,,
\end{array} \right.
$$
for some $k\in \{1,\dots,n-1\}$ and  with $c_{rs}^j, c_{r \overline {s}}^j \in \C$.
\end{lemma}

Now we can prove Theorem \ref{main}

\begin{proof}[Proof of Theorem $\ref{main}$]
Suppose that $(M=G/\Gamma, J)$ is not a complex torus, i.e. that $G$ is not abelian
and denote by $\mathfrak g$ the Lie algebra of $G$.  Assume that $(M,J)$ admits a balanced metric $g$ and also an SKT metric $g'$ both compatible with $J$. 
Then in view of Lemma \ref{2.1}, we may assume both $g$ and $g'$ invariant and regarding them as scalar products on $\mathfrak g$. This allows us to work at the level of the Lie algebra $\frak g$. 
As a consequence of Lemma \ref{2.3}, the existence of the {\rm SKT} metric implies that Lie algebra  $\frak g$ is $2$-step nilpotent and  that $(\frak g,J)$ has a $(1,0)$-coframe 
$\{\alpha^1, \ldots ,\alpha^n\}$ satisfying  the following structure equations 
\begin{equation}\label{streq}
\left \{ \begin{array}{l}
d \alpha^j= 0, \quad j = 1, \ldots, k,\\
d \alpha^j  = \sum_{r, s =1}^k  \left(\frac12 c_{rs}^j  \alpha^{r s} + c_{r \overline {s}}^j   \alpha^{r{\overline{s}}}\right).
\end{array} \right.
\end{equation}
for some $k\in \{1,\dots,n-1\}$ and with $c_{rs}^j, c_{r \overline s}^j \in \C$. Here we use the notation $\bar \alpha^i=\alpha^{\bar i}$ and $\alpha^{r_1\cdots r_p\bar s_1\cdots \bar s_q}=
\alpha^{s_1}\wedge \dots \alpha^{s_p}\wedge \alpha^{\bar s_1}\wedge \dots \alpha^{\bar s_q}$.
We may assume without restrictions that   the coframe $\{\alpha^i\}$ is unitary with respect to the balanced metric $g$. Indeed,  
since in the Gram-Schmidt process  the spaces spanned by the first $r$ elements of the original basis are preserved, we can modify the coframe $\{\alpha^r\}$ making it unitary with respect to $g$ and satisfying the same structure equations as in \eqref{streq} with different structure constants. 
In this way $g$ writes as 
$$
g=\sum_{r=1}^n\alpha^r\otimes \alpha^{\bar r}
$$
and the balanced condition can be written in terms of $c_{r\bar s}^l$'s as 
\begin{equation}\label{balanced}
\sum_{r=1}^k c_{r \ov r}^ l =0,
\end{equation}
for every $l > k$ (see also \cite[Lemma 2.1]{AV}). 

Next we focus on the SKT metric $g'$. Since  $\mathfrak g$ is $2$-step nilpotent, we have 
 $\partial\bar \partial \alpha^r=0$, $r=1,\dots,n$.
If $$
\omega'=\sum_{i,j=1}^n a_{i\bar j} \alpha^i\wedge \alpha^{\bar j}
$$
is the fundamental form of $g'$, then the SKT condition $\partial \bar \partial \omega'=0$ writes as 
\begin{equation}\label{SKT}
\sum_{i,j=k+1}^n    a_{i \overline j}  ( \overline \partial \alpha^i \wedge   \partial  \alpha^{\ov j}  - \partial \alpha^i \wedge \overline  \partial \alpha^{\ov j}) =0\,. 
\end{equation}
Equation \eqref{SKT} can be written in terms of the  structure constants as 
$$
\sum_{i,j=k+1}^n \sum_{r, s,u,v =1}^k   a_{i \overline j} \left(c_{r\bar v}^i \bar c_{u\bar s}^{ j}+\frac14c_{rs}^i\bar c_{uv}^{j}\right)   \alpha^{rs \bar u\bar v}  =0.
$$
By  considering the component along $ \alpha^{rs \bar s\bar r}$ in the above expression we get 
$$
\sum_{i,j=k+1}^n  a_{i \overline j}\left( c_{r\bar r}^i \bar c_{s\bar s}^{ j}+\frac14c_{rs}^i\bar c_{sr}^{j}
-c_{s\bar r}^i \bar c_{s\bar r}^{ j}-\frac14c_{sr}^i\bar c_{sr}^{j}
+c_{s\bar s}^i \bar c_{r\bar r}^{ j}+\frac14c_{sr}^i\bar c_{rs}^{j}
-c_{r\bar s}^i \bar c_{r\bar s}^{ j}-\frac14c_{rs}^i\bar c_{rs}^{j}\right)=0\,,
$$
i.e. the condition
\begin{equation} \label{coefficient}
\sum_{i,j=k+1}^n  a_{i \overline j}\left(c_{r\bar r}^i \bar c_{s\bar s}^{ j}
-c_{s\bar r}^i \bar c_{s\bar r}^{ j}+c_{s\bar s}^i \bar c_{r\bar r}^{ j}
-c_{r\bar s}^i \bar c_{r\bar s}^{ j}+c_{rs}^i\bar c_{sr}^{j}\right)=0\,.
\end{equation}
Now by taking the sum of  \eqref{coefficient} for $r,s=1,\dots k$ and keeping in mind the balanced assumption \eqref{balanced} we get 
\begin{equation}\label{SKTnew}
\sum_{i,j=k+1}^n \sum_{r,s=1}^k  a_{i \overline j}\left(
2c_{s\bar r}^i \bar c_{s\bar r}^{ j}
+c_{rs}^i\bar c_{rs}^{j}\right)=0\,.
\end{equation}
Let us consider now the $(1,0)$-vectors $X_{sr}$ and $X_{r\bar s}$ on $\mathfrak g$ defined by 
$$
X_{rs}=\sum_{i=k+1}^nc_{rs}^iX_i\,,\quad X_{r\bar s}=\sum_{i=k+1}^n  \sqrt{2}\,  c_{r\bar s}^iX_i
$$
where $\{X_1,\dots,X_n\}$ is the dual frame to  $\{\alpha^1,\dots,\alpha^n\}$. 
We have 
$$
\sum_{r,s=1}^k\left(\omega'(X_{rs},\bar X_{rs})+ \omega'(X_{r\bar s},\bar X_{r\bar s})\right)=\sum_{i,j=k+1}^n \sum_{r,s=1}^k  a_{i \overline j}\left(
2c_{s\bar r}^i \bar c_{s\bar r}^{ j}
+c_{rs}^i\bar c_{rs}^{j}\right)
$$
and so equation \eqref{SKTnew} implies $X_{rs}=X_{r\bar s}=0$ for every $r,s=1,\dots, k$. So every $\alpha^r$ is closed and $\mathfrak g $ is abelian, contradicting the first assumption in the proof. 
\end{proof}

\begin{remark}
{\em Theorem \ref{main} is trivial in dimension $6$. Indeed, in view of \cite{FPS} and Lemma \ref{2.1}, if a $6$-dimensional complex nilmanifold  $(M, J)$ has an SKT metric, then every  invariant   $J$-Hermitian  metric on $(M,J)$  is SKT and so every  invariant  balanced metric compatible with $J$ is automatically K\"ahler.  More generally, the same argument works when $M$ has arbitrary  real dimension $2n$, but $k=n-1$. }
\end{remark}

\begin{remark}{\em 
Another case where Theorem \ref{main} is trivial is when $k=1$. In this case  the Lie algebra  $\frak g$ is necessary isomorphic to $\frak h_3^{\R} \oplus \R^{2n - 3}$, where  $\frak h_3^{\R}$ is the real $3$-dimensional Heisenberg algebra,  and the result follows.
}

\end{remark}

\noindent {\it Acknowledgements}. We would like to thank Luis Ugarte and Raquel Villacampa for useful comments on the paper.

\end{document}